\documentclass[12pt]{article}
\usepackage{amsmath,amsthm,amsfonts,amssymb}
\usepackage{latexsym}
\usepackage{graphicx}
\usepackage{enumitem}
\usepackage{array,float}

\usepackage{graphicx,latexsym}
\usepackage{subfig}
\usepackage{cite}
\usepackage{hyperref}

\newtheorem{thm}{Theorem}[section]

\newtheorem{lem}[thm]{Lemma}
\newtheorem{prop}[thm]{Proposition}

\begin{document}

\title{Metric Dimensions of Generalized Sierpi\'{n}ski Graphs over Squares}

\author{S. Prabhu $^{a,}$\thanks{Corresponding author: drsavariprabhu@gmail.com}
	\and T. Jenifer Janany $^{a}$
	\and Sandi Klav\v zar $^{b,c,d}$
}
\maketitle

\begin{center}
	$^a$ Department of Mathematics, Rajalakshmi Engineering College, Thandalam, Chennai 602105, India\\
	\medskip
	
	$^b$ Faculty of Mathematics and Physics, University of Ljubljana, Slovenia\\
	\medskip

	$^c$ Institute of Mathematics, Physics and Mechanics, Ljubljana, Slovenia\\
	\medskip
	
	$^d$ Faculty of Natural Sciences and Mathematics, University of Maribor, Slovenia\\
	\medskip
\end{center}

\begin{abstract}
Metric dimension is a valuable parameter that helps address problems related to network design, localization, and information retrieval by identifying the minimum number of landmarks required to uniquely determine distances between vertices in a graph. Generalized Sierpi\'{n}ski graphs represent a captivating class of fractal-inspired networks that have gained prominence in various scientific disciplines and practical applications. Their fractal nature has also found relevance in antenna design, image compression, and the study of porous materials. The hypercube is a prevalent interconnection network architecture known for its symmetry, vertex transitivity, regularity, recursive structure, high connectedness, and simple routing. Various variations of hypercubes have emerged in literature to meet the demands of practical applications. Sometimes, they are the spanning subgraphs of it. This study examines the generalized Sierpi\'{n}ski graphs over $C_4$, which are spanning subgraphs of hypercubes and determines the metric dimension and their variants. This is in contrast to hypercubes, where these properties are inherently complicated. Along the way, the role of twin vertices in the theory of metric dimensions is further elaborated.

\end{abstract}

	
	
\section{Introduction}
Parallel computing is a type of computation in which many calculations or processes are carried out simultaneously. This is in contrast to serial computing, where calculations are performed sequentially, one after the other. Parallel computing offers several advantages, including improved performance, scalability, and the ability to tackle larger and more complex problems \cite{ClThMe06}. However, designing and implementing parallel algorithms can be challenging due to issues such as synchronization, load balancing, and communication overhead \cite{EdAl00}. Graph theory can have a significant impact on parallel computing, particularly in the context of algorithm design and optimization \cite{Ch86}. Graph theory provides a powerful framework for modeling and analyzing the structure and relationships of data, which is essential in many parallel computing applications. Many parallel computing applications involve graph algorithms, which operate on graphs to solve various problems such as shortest path finding, clustering, graph traversal, and network flow optimization \cite{Ch83}. Parallelizing these algorithms efficiently requires an understanding of both the graph structure and the characteristics of parallel computation.

A parallel computing network can indeed be viewed and represented as a graph structure. This representation can provide insights into the topology of the network, the communication patterns between computing nodes, and the distribution of computational tasks \cite{LuTu09}. Each computing unit in the parallel computing network, whether it's a processor core, a compute node, or a server, can be represented as a node in the graph. These nodes represent the individual processing elements that perform computations. The communication links between computing nodes are represented as edges in the graph. These edges depict the connectivity between nodes and can represent various types of communication channels, such as direct interconnects, network links, or shared memory connections. The arrangement of nodes and edges in the graph represents the network's topology. This includes characteristics such as whether the network is a clustered architecture, a mesh, a torus, a hypercube, or another topology  \cite{Du90}. The choice of topology can impact factors like communication latency, bandwidth \cite{MaVaCu97}, and fault tolerance. Graph-based representations can also be used for performance analysis and optimization of parallel computing networks. Techniques such as graph partitioning, load balancing, and routing algorithms \cite{BoHo82} can be applied to improve the efficiency and scalability of parallel computations. By representing a parallel computing network as a graph structure, we can gain insights into its characteristics, connectivity, and performance, which can aid in the design, analysis, and optimization of parallel algorithms and systems.

Graphic structures with self-similarity and recurrence are called fractals. Networks with fractal nature seems to be beneficial in the study of larger networks found in artificial and natural systems, such as neuroscience \cite{FeJe01}, music \cite{LoLo12}, social networking sites, and computers, allowing the subject of network science to progress. Images of complicated structures, such as neuronal dendrites or bacterial growth patterns \cite{ObPfSe90} in culture, can be captured and analysed with the help of such networks. Sierpi\'{n}ski networks are one of that class and can be used in parallel computing systems, particularly in the design of parallel supercomputers and high-performance computing (HPC) clusters. The Sierpi\'{n}ski network is characterized by its recursive and self-similar structure. It consists of interconnected nodes arranged in a hierarchical manner, with each level of the hierarchy representing a smaller version of the overall network. The Sierpi\'{n}ski network is inherently scalable, allowing it to accommodate a large number of nodes or processors. The self-similar structure of the Sierpi\'{n}ski network enables efficient routing and communication patterns. Due to these characteristics, the Sierpi\'{n}ski network can be explored as a potential topology for parallel computing systems, particularly in research contexts where innovative network designs are investigated to improve scalability, performance, and fault tolerance in parallel computing environments. In a particular case when the base graph is $K_3$, the Sierpi\'{n}ski networks are the Tower of Hanoi graphs with 3 pegs. Adding an open link to the extreme vertices of Sierpi\'{n}ski graphs results in WK-recursive networks. As abstract graphs, these networks were introduced in 1988 as message passing architectures and are employed in VLSI implementation~\cite{VeSa88}. A decade later, WK-recursive networks were equipped with the Sierpi\'{n}ski labelings and named Sierpi\'{n}ski graphs~\cite{KlMi97}. From then on the notion of Sierpi\'{n}ski graphs coincides with these graphs equipped with the Sierpi\'{n}ski labeling. The latter has also made it possible to explore Sierpi\'{n}ski graphs in more depth. Let us list some of the related results. This family of networks is proved to be Hamiltonian and the length of geodesic between their vertices are given in \cite{KlMi97}. Some metric properties including the average eccentricity \cite{HiPa12}, connectivity \cite{KlZe18} and median \cite{BaChHi22} are discussed. Some topological descriptors of Siepi\'{n}ski graphs are availabe in \cite{ImSaGa17}. All shortest paths in Sierpi\'{n}ski graphs are discussed in \cite{HiHe14}. The review paper~\cite{HiKlZe17} from 2017 summarizes the results on Sierpi\'{n}ski graphs and related classes of graphs and also proposes a classification of Sierpi\'{n}ski-like graphs. 

The classical Sierpi\'{n}ski graphs as introduced in~\cite{KlMi97} use  complete graphs are their basic stones. Replacing complete graphs by arbitrary graphs, generalized Sierpi\'{n}ski graphs were introduced in~\cite{GrKoPa11}. The key idea is again to create a recursive process that generates a fractal-like pattern within the graph. Generalized Sierpi\'{n}ski graphs are thus a broader class of fractal graphs that extend the concept of Sierpi\'{n}ski graphs to various shapes and structures beyond just complete graphs.

Let $G$ be a graph and $r\ge 1$ a positive integer. Then the {\em generalized Sierpi\'{n}ski graph} $S_G^r$ is formally defined as follows. Set $V(G) = [n] = \{1,\ldots, n\}$. Then $V(S_G^r) = [n]^k$, that is, vertices of $S_G^r$ are vectors of length $r$, each coordinate being a vertex of $G$. Two vertices $x= x_r \ldots x_1$ and $y = y_r \ldots y_1$ of $S_G^r$ are by definition adjacent if the following conditions hold for some index $t \in [r]$: 
\begin{enumerate}
\item[(i)] $x_i = y_i$ for $i >t$, 
\item[(ii)] $x_t \neq y_t$ and $x_ty_t \in E(G)$, 
\item[(iii)] $x_i = y_t$, $y_i = x_t$ for $i <t$.
\end{enumerate}	
In Fig.~\ref{gs}, a graph $G$ with $V(G) = [5]$ and $S_G^2$ are illustrated. 

\begin{figure}[ht!]
	\centering
	\subfloat[]{\includegraphics[scale=0.75]{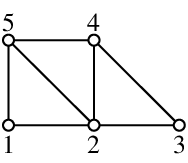}}
	\quad \quad \quad \quad \quad
	\subfloat[]{\includegraphics[scale=0.75]{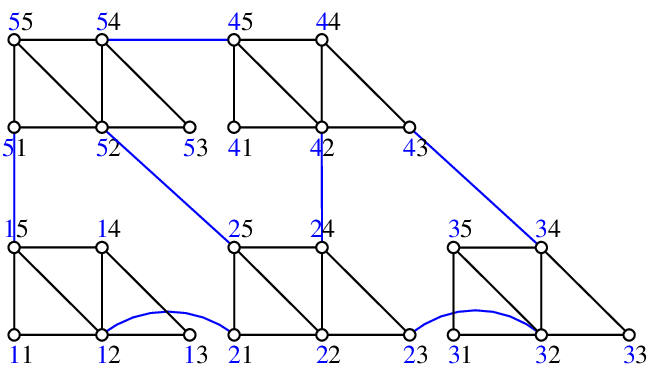}}
	\caption{(a) $S_{G}^1 \cong G$; (b) $S_{G}^2$}
	\label{gs}
\end{figure}

Generalized Sierpi\'{n}ski graphs are used to model polymer networks \cite{EsRo19}. Chromatic, vertex cover, clique and domination numbers \cite{RoRoEs17}, roman domination \cite{RaRoRo17}, strong metric dimension \cite{EsRo17} related results have been discussed for generalized Sierpi\'{n}ski graphs. In this article, we focus on the family $S_{C_4}^r$  and determine their metric dimension, edge metric dimension, fault-tolerant metric dimension and fault-tolerant edge metric dimension. The determination of these formulae is made possible by first examining in more detail the role of twin vertices in these dimensions.  

\section{The Problem and its Relevance}

Sensors can be used to detect faults or anomalies in the nodes or communication links of a parallel computing network. By continuously monitoring system parameters, sensors can help identify hardware failures, software errors, or performance bottlenecks, enabling proactive maintenance and troubleshooting. They can be deployed to detect unauthorized access, malicious activities, or security breaches in a parallel computing network. By monitoring network traffic, system logs, and environmental conditions, sensors can help identify and mitigate security threats in real-time. By collecting and analyzing sensor data, administrators and operators can make informed decisions to optimize system operation and ensure the efficient and reliable execution of parallel computing workloads. While sensor deployment can offer significant benefits in terms of system performance, reliability, and efficiency, it's essential to consider the associated costs, complexities, and trade-offs involved. Knowledge about metric dimension helps in optimizing sensor deployment.

\subsection{Basis and Fault-tolerant Basis}

Graph theory considers networks as topological graphs with interesting characteristics.
Metric dimension is a measure of how efficiently one can locate and distinguish between the vertices (nodes) of a graph using a minimal set of landmarks or reference points \cite{Sl75, HaMe76}. Metric dimension has applications in various fields, including network design \cite{BeEbEr06}, robotics \cite{KhRaRo96}, and location-based services. It helps in understanding how to place sensors or landmarks in a network to ensure efficient location determination \cite{ChZh03, ChErJo00, Jo93, MeTo84, LiAb06}. Calculating the exact metric dimension of a graph is often a computationally challenging problem. For certain classes of graphs, such as trees, there are efficient algorithms to find the metric dimension. However, for general graphs \cite{KhRaRo96}, bipartite graphs \cite{MaAbRa08} and directed graphs \cite{RaRaCy14} determining the metric dimension is NP-hard. Despite the computational difficulty, the precise value of metric dimension is evaluated for many graph structures including honeycomb \cite{MaRaRa08}, butterfly \cite{MaAbRa08}, Bene\v{s} \cite{MaAbRa08}, Sierpi\'{n}ski \cite{KlZe13}, and irregular triangular networks \cite{PrJeAr23}. The works related to metric dimension have been surveyed recently in \cite{TiFrLl23}. 

The distance $d_G(u,v)$ between two vertices in a connected graph $G$, with vertex set $V(G)$ and edge set $E(G)$, is equal to the number of edges in a geodesic (shortest path) connecting them. For a vertex $u$ and an edge $e= vw$ the distance between them is given by, $d_G(u,e) = \min \{d_G(u,v), d_G(u,w)\}$. For an ordered subset $X = \{u_1, u_2, \ldots, u_l\}$ of vertices, every vertex $x$ of $G$ can be represented by a vector of distances 
$$ r(x| X) = (d_G(x,u_1), d_G(x, u_2), \ldots, d_G(x, u_l))\,.$$ 
The subset $X$ is a {\em metric generator} (\textbf{MG}) if $r(x|X) \neq r(y|X)$, $\forall x,y \in V(G)$. In other words, $X$ is a \textbf{MG} if every pair $x,y \in V(G)$ has at least one vertex $u \in X$ such that $d(u,x) \neq d(u,y)$, see Fig.~\ref{rsdefn}. 

\begin{figure}[ht!]
	\centering
	\includegraphics[scale=.55]{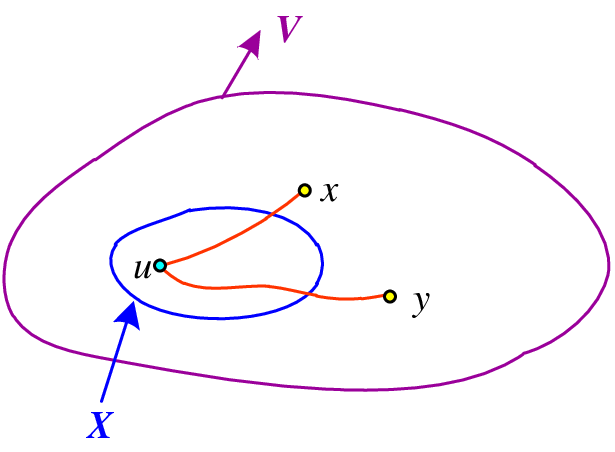}
	\caption {Metric generator (\textbf{MG})}
	\label{rsdefn}
\end{figure}	

A \textbf{MG} with minimum vertices is a ({\em metric}) {\em basis}; its cardinality is indicated by $\dim(G)$, which is the {\em metric dimension}. This concept is well-established in the literature, and has given rise to many different variations \cite{AlSuDa22, CyRaPr23, KeTrYe18, KlKu23}. One of these novel and highly motiated variations is known as fault-tolerant metric dimension. In this variation, the crucial aspect is that the selected set of vertices must still resolve the graph even when any one vertex from that set has become faulty or useless. That is, a set of vertices $F$ is termed as a {\em fault-tolerant metric generator} (\textbf{FTMG}) if, for every vertex $u$ within $F$, the set obtained by removing $u$ from $F$ remains a \textbf{MG} for the graph. In other words, even in the absence of any single vertex in the set $F$, the remaining vertices should still have the ability to resolve the graph.
\begin{figure}[ht!]
	\centering
	\includegraphics[scale=.55]{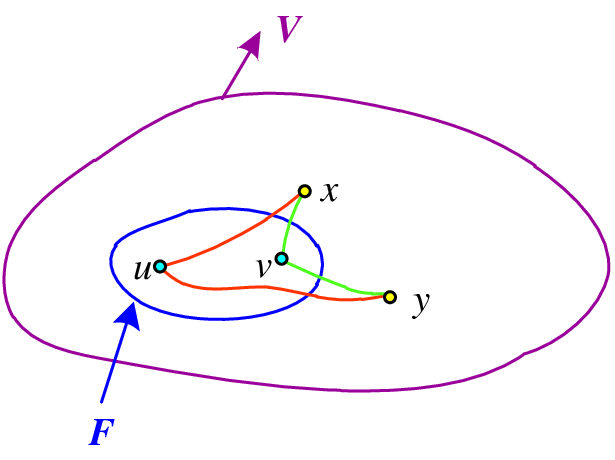}
	\caption {Fault-tolerant metric generator (\textbf{FTMG})}
	\label{ftrsdefn}
\end{figure}
The term {\em fault-tolerant metric dimension} is denoted as $\dim'(G)$, and it represents the minimum number of vertices required to form a \textbf{FTMG} for the graph $G$. The set of vertices that achieves this minimum and serves as a \textbf{FTMG} is referred to as a {\em fault-tolerant metric basis}. \textbf{FTMG} can also be described as a set of vertices $F$, with the property that for every pair of vertices $x$ and $y$ in the graph $G$, there exist at least two vertices $u$ and $v$ in set $F$, where $u$ and $v$ are at different distances from both $x$ and $y$, see Fig.~\ref{ftrsdefn}. 
A \textbf{FTMG} ensures that even if one vertex from the set $F$ is missing, there will still be another vertex in $F$ that can resolve $x$ and $y$ in the graph.  This idea was first presented in \cite{HeMoSl08} and was further discussed in \cite{RaHaPa19, BaSaDa20, PrMaAr22, ArKlPr23}.

\subsection{Edge Basis and Fault-tolerant Edge Basis}

In a parallel computing system, data is transmitted between processors through an interconnection network, which is a complex arrangement of processors and communication links. Identifying and addressing the links (connections) within this network is crucial for quickly pinpointing any faulty connections. This task can be performed optimally by utilizing the smallest possible set of vertices that uniquely label every edge in the network. In this context, a set of vertices $S$ within a graph $G$ is referred to as an {\em edge metric generator} (\textbf{EMG}) if it satisfies the condition that for any two edges $e$ and $f$ in the graph $G$, there exists a vertex $w$ in set $S$ such that $d_G(w,e)$ and $d_G(w,f)$ are distinct, see Fig.~\ref{ersdefn}.

\begin{figure}[ht!]
	\centering
	\includegraphics[scale=.55]{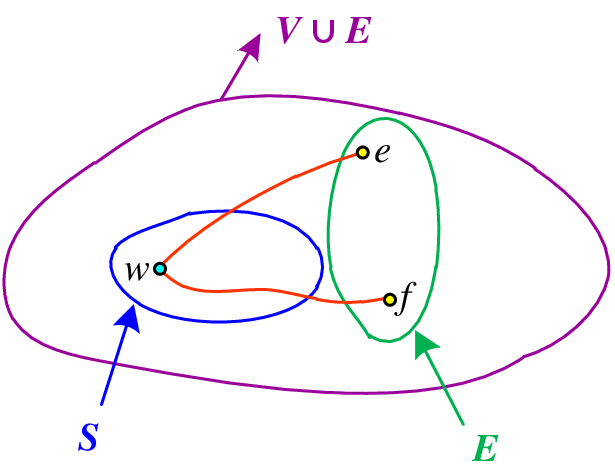}
	\caption {Edge metric generator (\textbf{EMG})}
	\label{ersdefn}
\end{figure}

An \textbf{EMG} ensures that each edge in the network is uniquely identifiable based on its distance to the vertices in set $S$. The number of vertices in the smallest possible \textbf{EMG} is referred to as the {\em edge metric dimension} and is connoted as $\dim_E(G)$. This variation was initially studied by Kelenc et al. \cite{KeTrYe18}, and they established its NP-completeness. Following the inception of this concept, numerous articles have emerged in this research field. To list a few, we have the characterization of graphs with maximum dimension~\cite{Zu18,ZhTaSh19}, the dimension of web graph, prism related graph, convex polytope graph \cite{ZhGa20}, generalized Petersen graph \cite{WaWaZh22}, some classes of planar graph \cite{WeSaSh21}, Erd\H{o}s-Renyi random graph \cite{Zu21}, graph operations such as join, lexicographic, corona \cite{PeYe20}, and hierarchical products \cite{KlTa21} for some graph classes. Identifying graphs with $\dim_E < \dim$ has gained more interest \cite{KnSkYe22, KnMaTo21}.

In order for an edge metric generator, denoted as $F$, to be considered fault-tolerant, it must possess an additional property, which is that the set obtained by removing any vertex $v$ from $F$ must still be capable of resolving the edges in the graph $G$. The minimum number of vertices required to form a {\em fault-tolerant edge metric generator} (\textbf{FTEMG}) is known as the {\em fault-tolerant edge metric dimension}, denoted as $\dim'_E(G)$. As depicted in Fig.~\ref{ftersdefn}, in a \textbf{FTEMG} set $F$, for any pair of edges $e$ and $f$ in the graph $G$, there are at least two vertices $v$ and $w$ within the set $F$ that are at different distances from both $e$ and $f$. 

\begin{figure}[ht!]
	\centering
	\includegraphics[scale=.55]{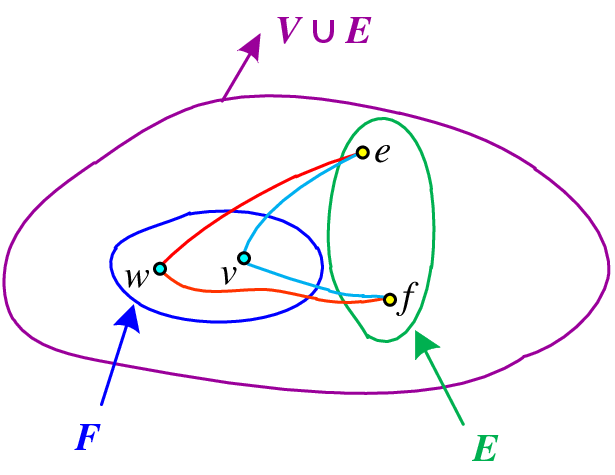}
	\caption {Fault-tolerant edge metric generator (\textbf{FTEMG})}
	\label{ftersdefn}
\end{figure}

\section{Effect of Twin Vertices on Metric Dimensions}
\label{sec:twins}
	
In this section, we will clarify how twins vertices affect all the variations of the metric dimension we are interested in. Some of the results on this are known from before, but others are being added newly. 

Let $G$ be a graph and $u\in V(G)$. Then $N_G(u) =\{v:\ uv \in E(G)\}$ is the open neighbourhood of $u$ and $N_G[u] = N(u) \cup \{u\}$ is the closed neighbourhood of $u$. A vertex $u$ is a {\em twin vertex} if there exists a vertex $v$ such that $N_G(u) = N_G(v)$ or $N_G[u] = N_G[v]$. In the first case we say that $u$ and $v$ are {\em non-adjacent twins} (also known as {\em false twins}), in the second case they are {\em adjacent twins} (also known as {\em true twins}). A maximal set of vertices $T$ in which every two vertices are twins, is called a {\em twin set}. These concepts are illustrated in Fig.~\ref{twindefn}. 

\begin{figure}[ht!] 
	\centering
	\subfloat[]{\includegraphics[scale=0.65]{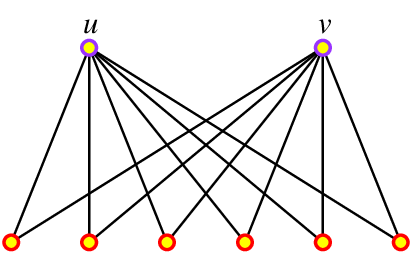}} 
	\quad \quad \quad   
	\subfloat[]{\includegraphics[scale=0.65]{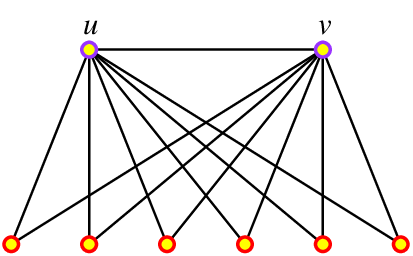}} 
	\quad \quad \quad   
	\subfloat[]{\includegraphics[scale=0.65]{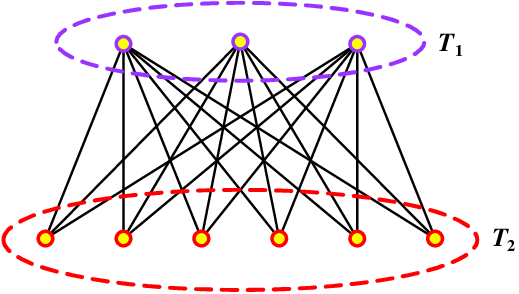}} 
	\caption{(a) Non-adjacent twin;  (b) Adjacent twin; (c) Twin sets } \label{twindefn}
\end{figure}
	
In~\cite[Corollary 2.4]{HeMoPe10} it was observed that if $S$ is a resolving set of a connected graph $G$ and $u$ and $v$ are twins, then $S\cap \{u,v\} \ne \emptyset$, cf.~\cite{PrFlAr18, PrMaAr22}. From this fact, the following conclusions can be easily drawn. 
	
\begin{prop} 
\label{mdlbd}
Let ${T}$ be the set of all twins of a connected graph $G$, and let $\mathcal{T}$ be the partition of ${T}$ into $k$ twin sets. Then the following properties hold. 
\begin{enumerate}[label=\rm {(\roman*)}]
\item {\rm \cite{PrFlAr18}} $\dim(G) \geq |{T}| - k$.
\item {\rm \cite{PrMaAr22}} $\dim'(G) \geq |{T}|$.
\item {\rm \cite{ArKlPr23}} If all twin sets are of cardinality $2$ and  $\dim(G) =k$, then $\dim'(G) =2k$.
\end{enumerate}
\end{prop}

In the case of bipartite graphs, (fault-tolerant) metric generators and (fault-tolerant) edge metric generators are closely related as follows, where the first property was proved in~\cite{KeToSk22}.   

\begin{prop} \label{emdubd} 
If $G$ is a connected bipartite graph, then the following properties hold.  
\begin{enumerate}[label=\rm {(\roman*)}]
\item 	A metric generator of $G$ is an edge metric generator of $G$.
\item 	A fault-tolerant metric generator of $G$ is a fault-tolerant edge metric generator of $G$. 
\end{enumerate}
\end{prop}

\begin{proof}
As said, (i) is due to~\cite{KeToSk22}. To prove (ii), consider any fault-tolerant metric generator $F$. Since $F \setminus \{u\}$ is a metric generator, it follows from (i)   that $F \setminus \{u\}$ is an edge metric generator. As this holds for every $u \in F$, the set $F$ happens to be a fault-tolerant edge metric generator.
\end{proof}

\begin{lem} \label{twincor}
If $S\subset V(G)$ is an edge metric basis of a connected graph $G$,  and $u$ and $v$ are twins, then $|S\cap \{u,v\}| \ge 1$. Also, if $u\in S$ and $v \notin S$, then $(S \setminus \{u\}) \cup \{v\}$ is another edge metric basis for $G$.
\end{lem}

\begin{proof}
The lemma follows from the fact that if $x$ is a common neighbor of $u$ and $v$, then $d_G(ux,w) = d_G(vx,w)$ holds for all $w \in V(G) \setminus \{u,v\}$.
\end{proof}

We next state a result for the (fault-tolerant) edge metric dimension parallel to Proposition~\ref{mdlbd}(i) and (ii).  

\begin{prop} 
\label{emdlbd}
If $G$ is a connected graph, $\mathcal{T}$ is the set of all twins in $G$, and $G$ has $k$ twin classes, then the following properties hold. 
\begin{enumerate}[label=\rm {(\roman*)}]
\item 	$\dim_E(G) \geq |\mathcal{T}|-k$.
\item $\dim'_E(G) \geq |\mathcal{T}|$.
\end{enumerate}
\end{prop}
	
\begin{proof}
(i) Let $S\subseteq V(G)$ be an edge metric basis of $G$ and let $\mathcal{T}_i$, $i\in [k]$, be the twin classes of $G$. Applying Lemma~\ref{twincor} to each pair of vertices from $\mathcal{T}_i$ we infer that $|S\cap \mathcal{T}_i| \ge |\mathcal{T}_i| - 1$. It follows that $\dim_E(G) \geq \sum_{i=1}^k (|\mathcal{T}_i|-1) = \left(\sum_{i=1}^k |\mathcal{T}_i|\right) - k = |\mathcal{T}|-k$. 
	
(ii) Let $S$ be a fault-tolerant edge metric generator of $G$. We claim that $\mathcal{T}\subseteq S$. Suppose on the contrary that there exists a twin vertex $u$ such that $u \notin S$. Let $v$ be a twin of $u$ and let $x$ be a common neighbor of $u$ and $v$. If $v\notin S$, then the distance between the edges $ux$ and $vx$ is the same to each vertex of $S$, which is clearly not possible. And if $v\in S$, then $S\setminus \{v\}$ is an edge metric generator, but then we get the same contradiction. This proves the claim and hence the assertion. 
\end{proof}

\section{Dimensions of $S_{C_4}^r$}

In this section we apply the result from the previous section to determine the four dimensions studied here of the generalized Sierpi\'nski graphs $S_{C_4}^r$, see Fig.~\ref{sc4} for $S_{C_4}^1$, $S_{C_4}^2$, and $S_{C_4}^3$. 

\begin{figure}[ht!]
	\centering
	\subfloat[]{\includegraphics[scale=0.85]{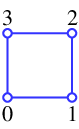}}
	\quad \quad
	\subfloat[]{\includegraphics[scale=0.85]{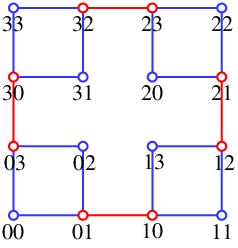}}
	\quad \quad
	\subfloat[]{\includegraphics[scale=0.85]{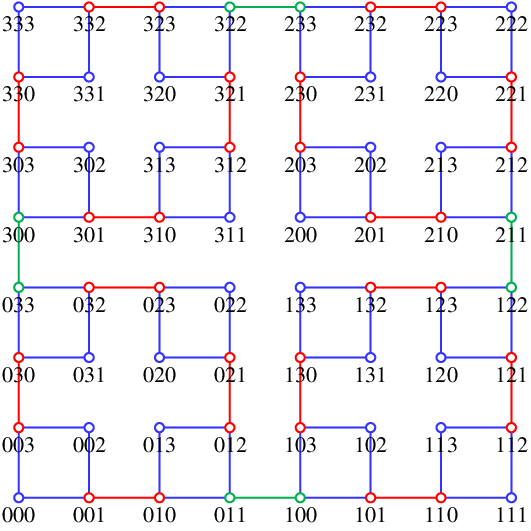}}
	\caption{(a) $S_{C_4}^1$; (b) $S_{C_4}^2$; (c) $S_{C_4}^3$}
	\label{sc4}
\end{figure}

Let $V(C_4) = \{0,1,2,3\}$ and let $r\ge 1$ be a fixed integer. Then, by the definition of the generalized Sierpi\'nski graphs $S_{C_4}^r$, the vertex set $V(S_{C_4}^r)$ partitions into subsets $V_i(S_{C_4}^r)$, $i\in \{0,1,2,3\}$, where the vertices of $V_i(S_{C_4}^r)$ start with $i$. That is, $V_i(S_{C_4}^r) = iV(S_{C_4}^{r-1})$, where we use the convention that if $X$ is a set of strings, then $iX$ is the set of strings derived by prefixing $i$ in each string from $X$. Moreover, for $r\ge 2$, the induced subgraphs $S_{C_4}^r[V_i(S_{C_4}^r)]$ are isometric subgraphs of $S_{C_4}^r$ and are isomorphic to $S_{C_4}^{r-1}$.

\begin{thm} \label{main}
If $r \geq 2$, then the following hold. 
\begin{enumerate}[label=\rm {(\roman*)}]
\item $\dim(S_{C_4}^r) = \dim_E(S_{C_4}^r) = \frac{4}{3}(2+4^{r-2})$. 
\item $\dim'(S_{C_4}^r) = \dim_E'(S_{C_4}^r)= \frac{8}{3}(2+4^{r-2})$.
\end{enumerate}
\end{thm}

\begin{proof}
Let $R_1 = \{0,1\}$ and for $k \geq 2$ set 
	\begin{align*}
		R_k =\ & 0R_{k-1} \setminus \{01^{k-2}1, 03^{k-2}1\}\ \cup \\ 
		      & 1R_{k-1} \setminus \{10^{k-2}0, 12^{k-2}0\}\ \cup \\
		      & 2R_{k-1} \setminus \{21^{k-2}1, 23^{k-2}1\}\ \cup \\ 
		      & 3R_{k-1} \setminus \{30^{k-2}0, 32^{k-2}0\}\,.
	\end{align*}
The first three of these sets are thus 
\begin{align*}
R_1 = & \{0,1\}, \\
R_2 = & \{00, 11, 20, 31\}, \\
R_3 = & \{000, 020, 111, 131, 200, 220, 311, 331\},
\end{align*}
while the vertices from the set $R_4$ are marked in Fig.~\ref{S4C4twins1}. 

\begin{figure}[ht!]
	\centering
	\includegraphics[scale=0.76]{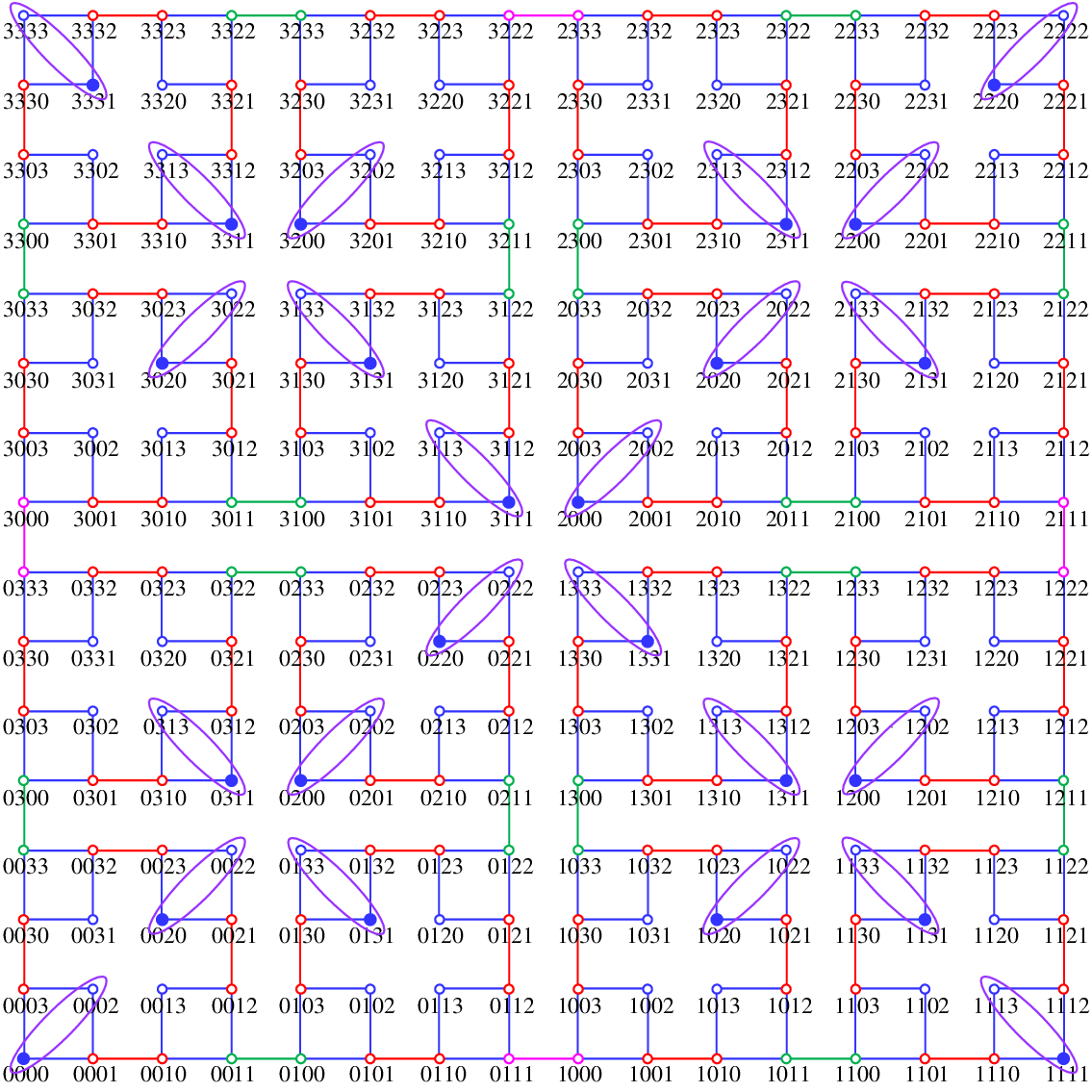}
	\caption {Twins in $S_{C_4}^4$ and the set $R_4$}
	\label{S4C4twins1}
\end{figure}

We claim that $R_k$ is a resolving set of $S_{C_4}^k$ and proceed by induction on $k$. Since $S_{C_4}^1 \cong C_4$ and $R_1 = \{0,1\}$, the claim holds true for $k=1$. We can also easily verify directly that the claim is true for $k\in \{2,3\}$. So it remains to verify that $R_{k+1}$, $k\ge 3$, is a resolving set of $S_{C_4}^{k+1}$. Consider arbitrary vertices $u = u_{k+1} \ldots u_1$ and $v = v_{k+1} \ldots v_1$ of $S_{C_4}^{k+1}$. 

Assume first that $u_{k+1} \neq v_{k+1}$. When $u_{k+1}=0$ and $v_{k+1} \in \{1,2,3\}$, then $d(u,x) > d(v,x)$ for $x \in \{2^{k}0, 2^{k}2\}$. Similarly, when $u_{k+1} = 1$ and $v_{k+1} \in \{0,2,3\}$, then $d(u,x) > d(v,x)$ for $x \in \{3^{k}1, 3^{k}3\}$. Also, if $u_{k+1} =2$, then $x \in \{0^{k}0, 0^{k}2\}$ and if $u_{k+1} =3$, then $x \in \{1^{k}1, 1^{k}3\}$ satisfies the inequality $d(u,x) >d(v,x)$. We can conclude that if $u_{k+1} \neq v_{k+1}$, then the vertices $u$ and $v$ are resolved by $R_{k+1}$. 

Assume second that $u_{k+1} = v_{k+1}$. We may without loss of generality assume that $u_{k+1} = v_{k+1} = 0$. By the induction assumption and the structure of $S_{C_4}^{k+1}$, the set $0R_k$ resolves $V_0(S_{C_4}^{k+1})$.  By the construction, $01^{k-1}1, 03^{k-1}1 \in 0R_{k}$ and $01^{k-1}1, 03^{k-1}1 \notin R_{k+1}$. Thus $r(01^{k-1}1 | R_{k+1} \cap V_0(S_{C_4}^{k+1})) = r(01^{k-1}3 | R_{k+1} \cap V_0(S_{C_4}^{k+1})) \neq r(03^{k-1}1 | R_{k+1} \cap V_0(S_{C_4}^{k+1})) = r(03^{k-1}3 | R_{k+1} \cap V_0(S_{C_4}^{k+1}))$. However, $d(01^{k-1}1, x) +2 = d(01^{k-1}3, x)$ and $d(03^{k-1}1, x) = d(03^{k-1}3, x) +2$ for every $x \in V_2(S_{C_4}^{k+1})$. We can thus conclude that $R_{k+1}$ is a resolving set of $S_{C_4}^{k+1}$. 

If $r\ge 2$, then clearly $|R_r| = 4(|R_{r-1}|-2)$. Resolving this recurrence relation we get $|R_r|= \frac{4}{3}(2+4^{r-2})$ for $r\geq 2$. Thus, $\dim(S_{C_4}^r) \leq \frac{4}{3}(2+4^{r-2})$. On the other hand, using induction we infer that a vertex $u$ of $S_{C_4}^{r}$, where $r\ge 2$, is a twin vertex if and only if its degree is $2$ and lies in a 4-cycle containing another vertex of degree $2$ (which is the twin of $u$). Moreover, each twin set is of cardinality $2$ and contains exactly one vertex of $R_r$; see Fig.~\ref{S4C4twins1} again. Hence by Proposition~\ref{mdlbd}(i) we have $\dim(S_{C_4}^r) \geq \frac{4}{3}(2+4^{r-2})$, hence we may conclude that $\dim(S_{C_4}^r) = \frac{4}{3}(2+4^{r-2})$. From here, and recalling the fact that each twin set is of cardinality $2$, Proposition~\ref{mdlbd}(iii) yields $\dim'(S_{C_4}^r) =  \frac{8}{3}(2+4^{r-2})$. 

As $S_{C_4}^r$ is bipartite, the already established fact that $\dim(S_{C_4}^r) = \frac{4}{3}(2+4^{r-2})$ together with Proposition~\ref{emdubd}(i) yields $\dim_E(S_{C_4}^r) \le \frac{4}{3}(2+4^{r-2})$. But then, analogously as in the above paragraph we get by Proposition~\ref{emdlbd}(i) that $\dim_E(S_{C_4}^r) = \frac{4}{3}(2+4^{r-2})$ and then by Proposition~\ref{emdlbd}(ii) and Proposition~\ref
{emdubd}(ii) that $\dim_E'(S_{C_4}^r) = \frac{8}{3}(2+4^{r-2})$. 
\end{proof}

\section{Comparison and Future Direction}

\begin{figure}[ht!] 
	\centering
	\subfloat[]{\includegraphics[scale=.85]{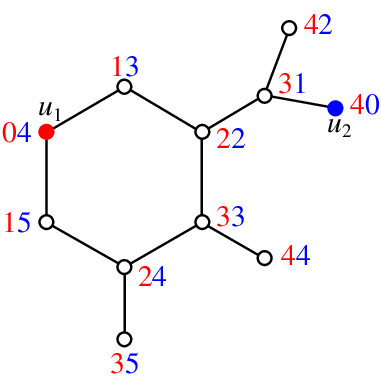}}
	\quad \quad \quad \quad
	\subfloat[]{\includegraphics[scale=.85]{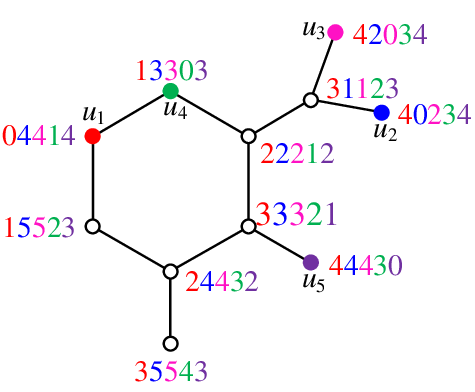}}
	\caption{(a) Metric basis with two elements;  (b) Fault-tolerant metric basis with five elements} \label{fteg}
\end{figure}

Assignment of tasks to efficient labour gives rise to yet another problem of job allocation. When a person is not able to perform the task assigned to him, it is very difficult to identify an efficient and equivalent skilful person to complete the task. This is sometimes easy if all of them are equally trained. Let us consider the situation of basis in a complete graph. If an element of the basis of a complete graph becomes faulty, then the node outside the basis can perform the task of the faulty one. This happens due to the availability of the node outside the basis. But, in general, we cannot expect this to happen on any basis of an arbitrary graph. On the other side of the problem, if one node of the basis is faulty, then precisely one node that is the best replacement for that node exists. If this happens for every member of the basis, then the fault-tolerant basis is twice that of the basis. This may motivate us to conclude $\dim'(G)\leq 2 \dim(G)$. But in reality, this is not true. Fig.~\ref{fteg} shows a graph $G$ with $\dim'(G) > 2 \dim(G)$. It is clear from the figure  that the node $u_2$ and node $u_3$ best replace each other whereas for the node $u_1$ we require $u_4$ and $u_5$ as the replacement nodes. Characterization of graphs which admits $\dim'(G) = 2 \dim(G)$ were reported in \cite{ArKlPr23} and the authors proved that this holds for a certain fractal cubic network. This paper investigates yet another graph class for which $\dim'(G) = 2 \dim(G)$. The graphical comparison on $\dim$, and $\dim'$ over the node set cardinality is depicted in Fig.~\ref{compare}(a). Similarly for $\dim_E$ and $\dim_E'$ over the edge set cardinality is depicted in Fig.~\ref{compare}(b).
\begin{figure}[ht!] 
	\centering
	\subfloat[]{\includegraphics[scale=0.34]{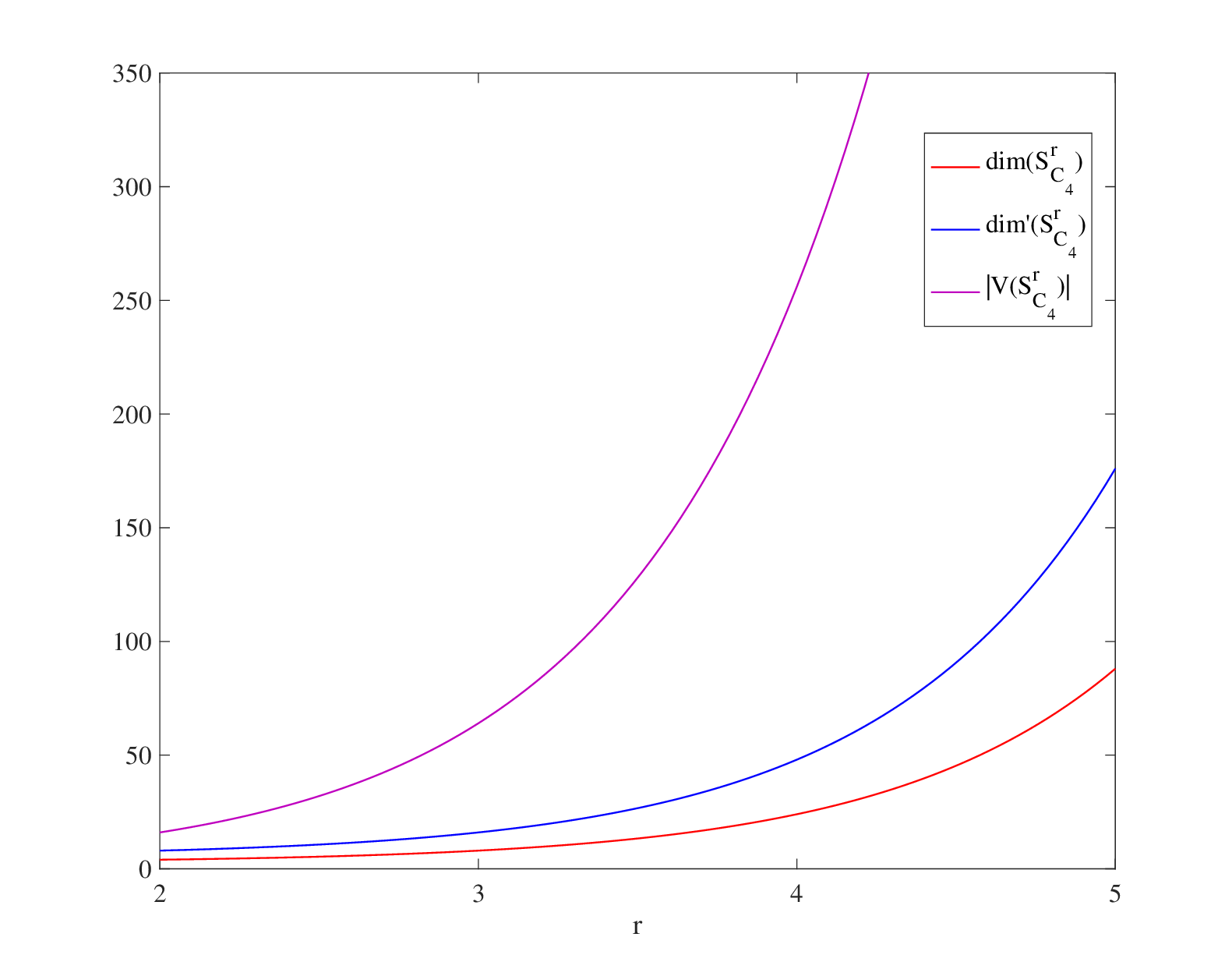}} \quad
	\subfloat[]{\includegraphics[scale=0.34]{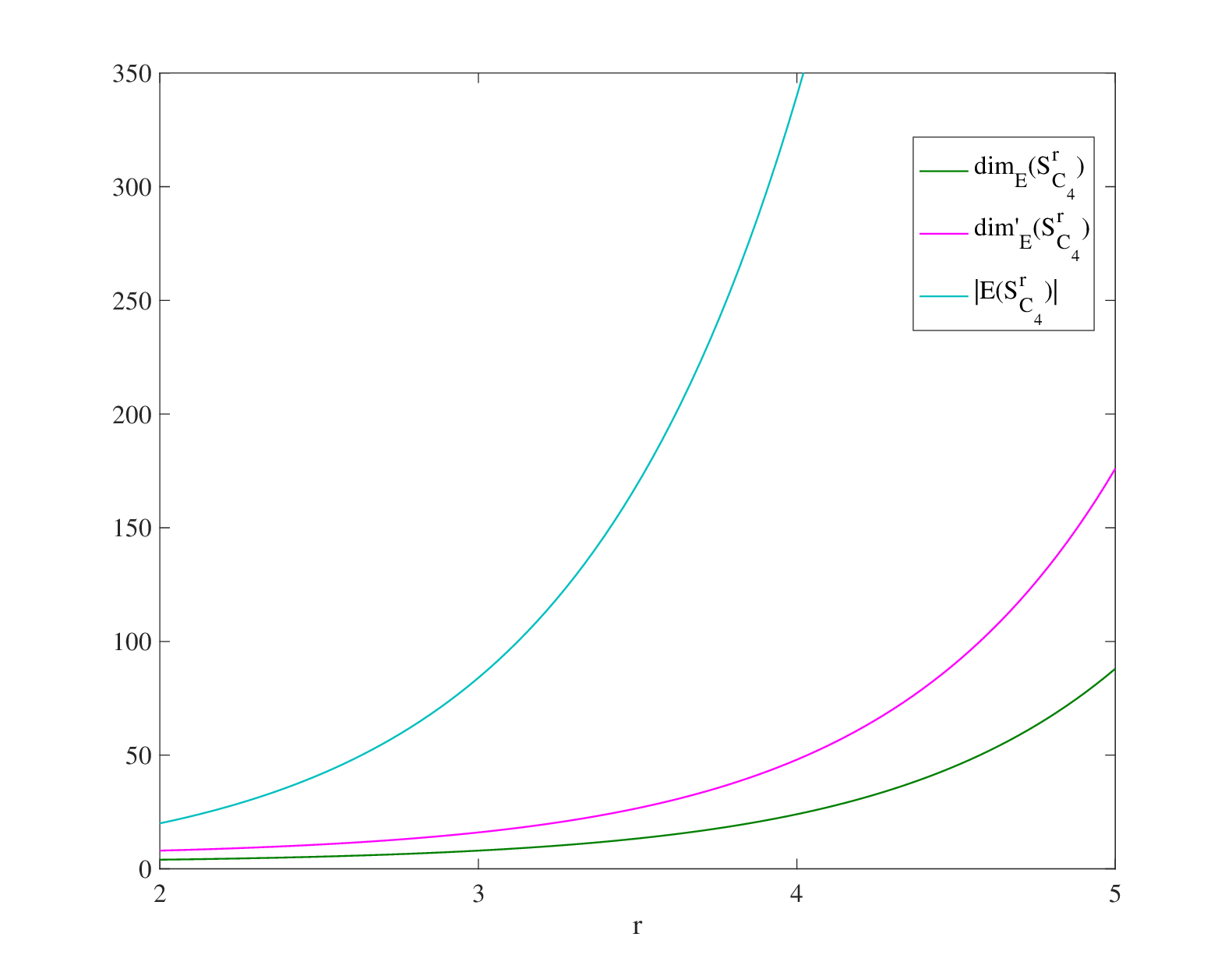}}
	\caption{(a) Comparison between the order of $S_{C_4}^r$, $\dim(S_{C_4}^r)$ and $\dim'(S_{C_4}^r)$;  (b) Comparison between the  size of $S_{C_4}^r$, $\dim_E(S_{C_4}^r)$ and $\dim'_E(S_{C_4}^r)$} \label{compare}
\end{figure}

\section{Conclusion}

The metric dimension of a graph is a measure of how many vertices (or nodes) are needed to uniquely identify all other vertices in the graph based on distance information. By identifying a small set of nodes that can uniquely identify all other nodes in the network, routing algorithms can be optimized to minimize message overhead and latency. By studying the metric dimension, researchers can analyze the network's resilience to node or link failures and develop strategies for fault recovery and rerouting in parallel computing systems. Insights gained from studying the metric dimension of a network can provide valuable insights into the network's characteristics and behavior, which can be leveraged to optimize routing algorithms, and enhance overall performance in parallel computing systems. In some cases, sensors may be placed in the nodes of a parallel computing network for various purposes, depending on the specific requirements of the application and the goals of the system. In this paper, the metric dimension, edge metric dimension and their fault-tolerant versions are investigated for the generalized Sierpi\'{n}ski graphs derived from $C_4$. This family of graphs have an intresting relation with the widely known interconnection network, hypercube. The $r$-dimensional generalized Sierpi\'{n}ski graph $S_{C_4}^r$ is a spanning subgraph of the $2r$ dimensional hypercube.

\section*{Acknowledgements}

Sandi Klav\v zar was supported by the Slovenian Research Agency ARIS (research core funding P1-0297 and projects J1-2452, N1-0285).

\end{document}